\documentclass[12pt]{amsart}
\usepackage{amsmath,amsfonts,amsthm,amssymb,mathrsfs}
\usepackage{times, graphicx}
\usepackage{color,ulem}
\usepackage{hyperref}
\theoremstyle{plain}
\newtheorem*{theorem*}{Theorem}
\newtheorem{theorem}{Theorem}

\newtheorem{lem}[theorem]{Lemma}
\newtheorem{prop}[theorem]{Proposition}
\newtheorem{proposition}[theorem]{Proposition}

\theoremstyle{definition}
\newtheorem{definition}[theorem]{Definition}

\theoremstyle{remark}

\newcommand{\vphi}{\varphi}

\newcommand{\pl}{\partial}

\newcommand{\lt}{\left}
\newcommand{\rt}{\right}
\newcommand{\rw}{\rightarrow}
\newcommand{\R}{\mathbb{R}}

\renewcommand{\tilde}{\widetilde}

\newcommand{\mbf}{\mathbf}

\title[Convex body isoperimetric conjecture]{A note on the convex body isoperimetric conjecture in the plane}
\author[]{Bo-Hshiung Wang and Ye-Kai Wang}
\address{Department of Mathematics, National Cheng Kung University}
\thanks{The second author is supported by MOST Taiwan grant 109-2628-M-006-001-MY3 and he would like to thank Professor Kwok-Kun Kwong for the discussions.}

\begin{document}
\begin{abstract}
The convex body isoperimetric conjecture in the plane asserts that the least perimeter to enclose given area inside a unit disk is greater than inside any other convex set of area $\pi$. In this note we confirm two cases of the conjecture: domains symmetric to both coordinate axes and perturbations of unit disk.
\end{abstract}

\maketitle
\section{Introduction}
Let $\Omega \subset \R^2$ be a bounded domain. For $0<A<\mbox{Area}(\Omega)$, consider the variational problem
\begin{align}\label{varitional problem}
I_\Omega(A) = \min \{ \mbox{Length}(\gamma) : \gamma \mbox{ encloses a region of area } A \mbox{ inside } \Omega \}.
\end{align}
The function $I_\Omega: (0,\mbox{Area}(\Omega)) \rw (0,\infty)$ is called the {\it isoperimetric profile} of $\Omega$. Note that \begin{align}
I_\Omega(A) = I_\Omega(\mbox{Area}(\Omega) - A).
\end{align}

The convex body isoperimetric conjecture in the plane asserts that if $\Omega$ is a convex domain of area $\pi$, then
\begin{align}
I_\Omega(A) \le I_{B_1}(A)
\end{align} for $0<A<\pi$.

We first learned this conjecture in M. Hutching's webpage \cite{Hutchings}. He attributes it to Wicharamala. F. Morgan's blog contains an extensive discussion of the conjecture \cite{Morgan}, mostly focusing on its higher dimensional version.

The conjecture is completely solved for $A = \frac{\pi}{2}$ by Esposito et. al. 
\begin{theorem}\cite[Theorem 1]{EFKNT}
If $K$ is an open convex set of $\R^2$, we have:
\[ \inf_{G \subset K, |G| = |K|/2} Per(G;K)^2 \le \frac{4}{\pi}|K|. \]
Moreover, equality holds if and only if $K$ is a disk. 
\end{theorem}
In addition to being interesting on its own, the theorem leads to several relative isoperimetric inequalities. The conjecture also holds true for regular polygons \cite[Theorem 4.1]{Berry}. Besides these two cases, little is known.

In this note, we confirm the conjecture for two special cases. The first is
\begin{theorem}[Theorem \ref{bi-symmetric}]\label{bi-symmetric intro}
Let $\mathcal{A}$ be the class of domains $\Omega$ bounded by smooth convex curves that are symmetric in both coordinate axes and have exactly four vertices. Suppose that $\Omega \in \mathcal{A}$ has area $\pi$ and is not a unit disk. Then
\[ I_\Omega(A) < I_{B_1}(A) \]
for $0<A<\pi$.
\end{theorem}

The class $\mathcal{A}$ has been studied by \cite{AB}. It serves as a model for a comparison theorem of isoperimetric profile which leads to a new proof of Grayson-Gage-Hamilton Theorem for curve-shortening flow. The crucial fact about $\mathcal{A}$ is that the minimizers of  \eqref{varitional problem} for $\Omega \in \mathcal{A}$ admits a simple characterization.

The second result is obtained by analyzing the first and second variation of length under area-preserving perturbation. 
\begin{theorem}[Theorem \ref{perturbation}]\label{perturbation short}
The conjecture holds for perturbations of the unit disk.
\end{theorem} 

The paper is organized as follows. In section 2, we set up the notations and present background materials. Theorem \ref{bi-symmetric intro} and Theorem \ref{perturbation short} are proved in Section 3 and 4 respectively. We include two appendices describing relevant results to our main theorems.

\section{Preliminaries}
All domains considered in this note are assumed to be bounded, connected and have smooth boundary $\pl\Omega$ whose (signed) curvature is denoted by $\kappa$.

We start with the following well-known description of the minimizers of the variational problem \eqref{varitional problem}. 
\begin{prop}
A least-perimeter curve enclosing a given area within a region consists of circular arcs or straight line segments meeting the boundary orthogonally. Moreover, if the region is convex, then a least-perimeter curve is connected. 
\end{prop}
\begin{proof}
The first assertion is a consequence of the first variation of length. See Lemma 3.2 of \cite{AB} for example. The second assertion follows from an observation of Kuwert that $I_\Omega^2$ is concave for convex domains. See Section 2 of \cite{Berry}.
\end{proof}
\begin{definition}
A {\it perfect arc} in $\Omega$ is a circular arc or straight line segment (not necessary a minimizer of \eqref{varitional problem}) inside $\Omega$ that meets $\pl\Omega$ orthogonally.
\end{definition}
We have another consequence of the first variation formula.
\begin{prop}\label{dLdA}
Suppose $\gamma(t)$ is a family of perfect arcs with constant curvature $k(t)$. Let $L(t)$ and $A(t)$ denote the length and the enclosed area of $\gamma(t)$. Then
\[ \frac{dL}{dt} = k(t) \frac{dA}{dt}.\]
\end{prop}
Next, we recall the fundamental Pestov-Ionin ineqaulity.
\begin{theorem}\label{Pestov-Ionin}
If $\gamma$ is a simple closed smooth curve, then \[ \kappa_{\max} \ge \sqrt{\pi/A},\]
where $\kappa_{\max}$ is the maximum curvature and $A$ is the enclosed area. The equality holds if and only if $\gamma$ is a circle.
\end{theorem}
The original proof \cite{Pestov} is not easily accessible. See the recent lecture note \cite{Petrunin}	for an elementary account or \cite{Pankrashkin} for a proof using curve shortening flow. We will only use a simple implication of the inequality: If $\Omega$ is a domain with area $\pi$ and is not a unit disk, then the maximum curvature of $\pl\Omega$ is greater than 1 and, as a result of $\int_{\pl\Omega} \kappa ds = 2\pi$ and the isoperimetric inequality, the minimum curvature is less than 1.

An immediate consequence is that the conjecture is true for sufficiently small $A$ without the convexity assumption.
\begin{theorem}\label{small}
Let $\Omega$ be a domain with area $\pi$. Then there exists $\delta >0$ such that
\[ I_\Omega(A) < I_{B_1}(A) \]
for $0 < A < \delta$.
\end{theorem}
The assertion follows from Proposition 2.1 of \cite{AB}:
\[ \lim_{a \rw 0} \frac{I_\Omega(a) - \sqrt{2\pi a}}{a} = - \frac{4 \max_{\pl\Omega} \kappa}{3\pi}.\] 
See also Section 5 of \cite{Berry} for a proof that applies to convex regions based on the analysis of regular polygons. 

\section{Symmetric domains}
We fix a domain $\Omega \in \mathcal{A}$. We assume that the major axis of $\Omega$ lies on the $x$-axis and the minor axis  lies on the $y$-axis. Write $C$ for $\pl\Omega$ and denote the unit tangent and unit outward normal of $C$ by $T$ and $N$. Since $C$ is convex, we can parametrize it by the normal vector. Namely, $C = C(\theta)$ with $N = (\cos\theta, \sin\theta)$.

The following three lemmas were obtained in Section 4 of \cite{AB}. We present an elementary proof of them.
\begin{lem}\label{CdotT}
We have $C \cdot T <0$ when $0 < \theta < \pi/2$. In particular, the maximum radius of $\Omega$ are attained at the vertices on the $x$-axis and the minimum radius are attained at the vertices on the $y$-axis.
\end{lem}
\begin{proof}
Denote $\frac{df}{d\theta} = f'$ in the proof. Direct computation yields $(C \cdot N)' = C \cdot T$ and $(C \cdot T)' = \frac{1}{\kappa}- C \cdot N $. Write $\Psi = -C \cdot N$ and we get $\Psi' + \Psi''' = \frac{\kappa'}{\kappa^2} < 0$ on $(0,\frac{\pi}{2})$.

Let 
\begin{align*}
P &= \Psi' \cos\theta - \Psi'' \sin\theta\\
Q &= \Psi' \sin\theta + \Psi'' \cos\theta.
\end{align*}
Then $P(0) = Q(\frac{\pi}{2})=0$. Moreover, we have $P' = -\sin\theta (\Psi' + \Psi''') > 0$ and $Q' = \cos\theta (\Psi' + \Psi''') <0$ on $(0,\frac{\pi}{2})$. It follows that $P,Q,$ and $\Psi' = - C \cdot T = P \cos\theta + Q\sin\theta$ are all positive on $(0,\frac{\pi}{2})$. 
\end{proof}

\begin{lem}
The perfect arc that is symmetric with respect to the $x$-axis is contained inside $\Omega$.
\end{lem}
\begin{proof}
Suppose the end points of the perfect arc are $C_1 = C(\theta)$ and $C_2$. Then the normal lines at $C_1$ and $C_2$ intersect at $p = (- \frac{C(\theta)\cdot T(\theta)}{\sin\theta},0).$ Since $p \in \Omega$ and the perfect arc is contained in the triangle $pC_1C_2$, the perfect arc is contained in $\Omega$ by the convexity.  
\end{proof}

We parametrize the family of perfect arcs that are symmetric with respect to the $x$-axis by $\theta$ and denote their length by $L(\theta)$ and enclosed area by $A(\theta)$. 
\begin{lem}
Both $L(\theta)$ and $A(\theta)$ are strictly increasing functions in $\theta$.
\end{lem}
This is not obvious as the perfect arcs may cross each other.

\begin{figure}[h]
\caption{Intersection of perfect arcs. Readers should smooth out the corners.}
\centering
\includegraphics[scale=0.4]{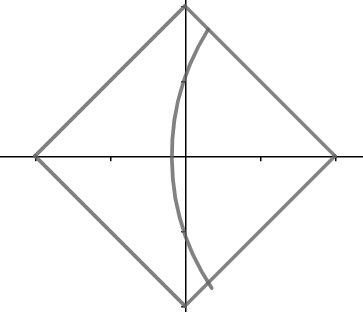}
\end{figure}
\begin{proof}

Suppose $C(\theta) = (x(\theta),y(\theta))$. Then we have $y(\theta) = \int_0^\theta \frac{\cos\omega}{\kappa(\omega)}\,d\omega$ and $L(\theta) = \frac{\pi - 2\theta}{\cos\theta} \int_0^\theta \frac{\cos\omega}{\kappa(\omega)}\,d\omega$. Since $\kappa$ is decreasing on $(0,\frac{\pi}{2})$, the derivative of $L$ on $(0,\frac{\pi}{2})$ satisfies
\begin{align*}
\frac{dL}{d\theta} &= \frac{(\pi - 2\theta)\sin\theta - 2\cos\theta}{\cos^2\theta} \int_0^\theta \frac{\cos\omega}{\kappa(\omega)}\,d\omega + \frac{\pi - 2\theta}{\cos\theta} \frac{\cos\theta}{\kappa(\theta)}\\
&\ge \frac{(\pi - 2\theta)\sin\theta - 2\cos\theta}{\cos^2\theta} \int_0^\theta \frac{\cos\omega}{\kappa(0)} d\omega + \frac{\pi - 2\theta}{\kappa(0)}\\
&= \frac{(\pi - 2\theta) - \sin 2\theta}{\kappa(0)\cos^2\theta} > 0.
\end{align*}

By Proposition \ref{dLdA}, we also get $\frac{dA}{d\theta} >0$.

\end{proof}

Now we further impose that $\Omega$ has area $\pi$ and is not a unit disk. We denote the length of perfect arcs in the unit disk by $L^*(\theta)$.
\begin{lem}\label{Lcomparison}
If $0<\theta \le \theta^* \le \pi/2$, then $L(\theta) < L^*(\theta^*)$.
\end{lem}
\begin{proof}
Since $L$ is increasing in $\theta$, it suffices to show that $L(\theta^*) < L^*(\theta^*)$. Moreover, it suffices to show that $y(\theta^*) < y^*(\theta^*)$ where $(x^*(\theta), y^*(\theta))$ is the upper endpoint of the perfect arc in the unit disk.

Let $s$ be the arclength parameter. We have $\frac{dy}{ds} = \cos\theta, \frac{d\theta}{ds} = \kappa$ and hence the relation $y(\theta) - y^*(\theta) = \int_0^{\theta} \cos\omega (\frac{1}{\kappa(\omega)} -1 )\,d\omega$, which implies that $y - y^*$ is decreasing on $(0,\bar\theta)$ and increasing on $(\bar\theta, \pi/2]$, where $\kappa(\bar\theta)=1$. 

Suppose $y(\theta^*) \ge y^*(\theta^*)$ for some $\theta^*$. Then $\theta^* \in (\bar\theta, \pi/2]$ and we infer that $y(\frac{\pi}{2}) \ge y^*(\frac{\pi}{2})$. Therefore the minor axis of $\Omega$ is longer than 2 and $\Omega$ contains a unit disk by Lemma \ref{CdotT}. This contradicts the assumption that $\Omega$ has area $\pi$.  
\end{proof}

\begin{figure}[h]
\caption{Illustration for Lemma \ref{Lcomparison}}\centering
\includegraphics[scale=0.4]{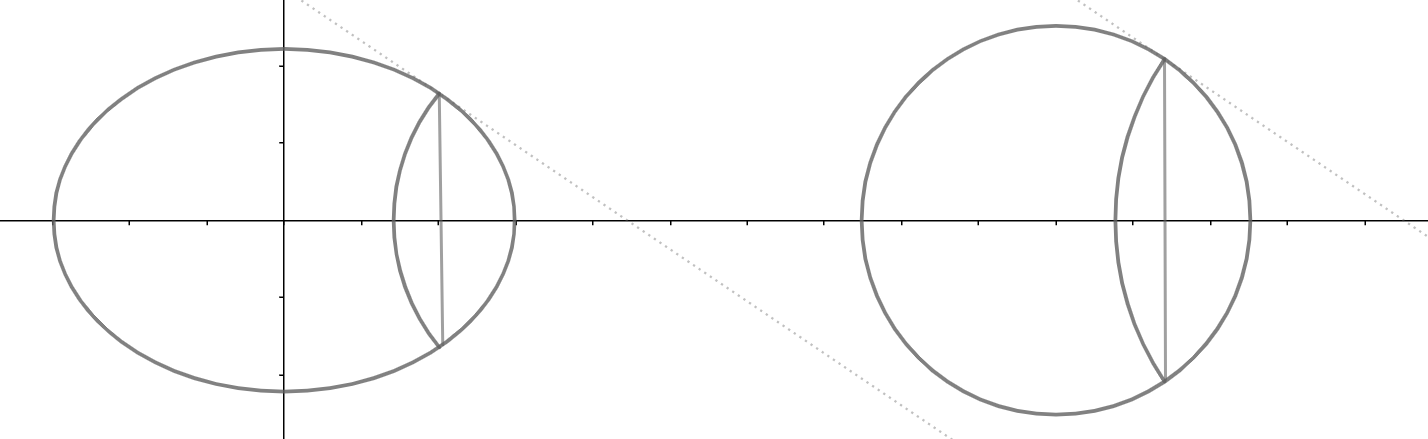}
\end{figure}

We are ready to prove the main theorem of this section.\begin{theorem}\label{bi-symmetric}
Let $\mathcal{A}$ be the class of domains $\Omega$ bounded by smooth convex curves that are symmetric in both coordinate axes and have exactly four vertices. Suppose that $\Omega \in \mathcal{A}$ has area $\pi$ and is not a unit disk. Then
\[ I_\Omega(A) < I_{B_1}(A) \]
for $0<A<\pi$.
\end{theorem}
\begin{proof}
Since $A(\theta)$ is strictly increasing on $(0,\frac{\pi}{2})$, we change variable to consider $L$ as a function of $A$. Since $I_\Omega(A) \le L(A)$ and $I_{B_1}(A) = L^*(A)$, it suffices to show that supremum of the function $\frac{L}{L^*}(A)$ on $A \in (0,\pi)$, which is symmetric with respect to $\pi/2$, is less than 1.  If the supremum occurs at $A =0$, then the assertion follows from Theorem \ref{small} so we assume the absolute maximum occurs at $\bar A \in (0, \frac{\pi}{2}]$. We have
\begin{align*}
0 = \lt( \frac{L}{L^*}\rt)'(\bar A) = \frac{k L^* - k^* L}{L^{*2}}(\bar A)
\end{align*}
Since $L = \frac{\pi - 2\theta}{k}$, we obtain
\[ \frac{\pi - 2\theta}{\pi - 2\theta^*} = \lt( \frac{L}{L^*} \rt)^2\]
at $\bar A$. Either $\theta > \theta^*$ or $\theta \le \theta^*$ leads to $L < L^*$, where Lemma \ref{Lcomparison} is used in the latter case.  
\end{proof}
 
\section{Perturbation of the unit disk}
\subsection{Prefect arcs in the unit disk}
We begin this section by describing the perfect arcs in the unit disk, following Section 2 of \cite{AB}. The isoperimetric profile of $B_1$ is given implicitly by
\[ I_{B_1}(a) = (\pi - 2\theta) \tan\theta, \quad a = \theta - \tan\theta + (\frac{\pi}{2} - \theta)\tan^2\theta. \]
The perfect arc $\tilde\sigma$ is given by
\begin{align*}
\tilde\sigma(x) = (\sec\theta,0) + \tan\theta \lt( \cos \lt( (\pi-2\theta)x + \frac{\pi}{2} + \theta\rt), \sin \lt( (\pi-2\theta)x + \frac{\pi}{2} + \theta \rt) \rt), \quad 0 \le x \le 1.
\end{align*} 
The other perfect arcs $\tilde \sigma(x;u)$ are obtained by rotating $\tilde\sigma$ counterclockwise by an angle of $u$.

\begin{figure}[h]
\caption{Perfect arcs of the unit disk.}
\centering
\includegraphics[scale=0.5]{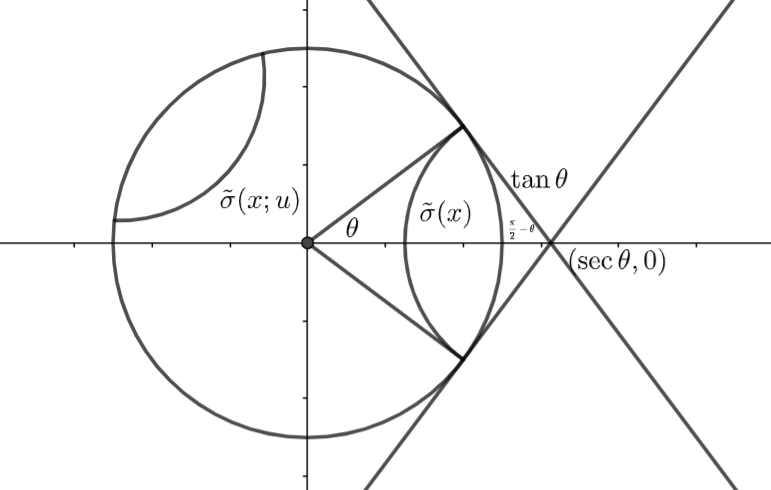}
\end{figure}

\subsection{Setup of the perturbative analysis}
We follow closely the convention in \cite{AB}. See their Figure 2 for an instructive summary. Let $X(u,s): [0,2\pi] \times [0,\delta) \rw \R^2$ be a perturbation of the unit circle that preserves area. Namely, denoting the domain enclosed by $X(\cdot,s)$ by $\Omega_s$, $X(u,s)$ satisfies $X(u,0) = (\cos u, \sin u)$ and $\mbox{Area}(\Omega_s) = \pi$.

We consider a family of curves $\sigma(x,s;u): [0,1] \times [ 0, \delta) \times [0,2\pi]$ inside $\Omega_s$ satisfying
\begin{enumerate}
\item $\sigma(x,0;u) = \tilde\sigma(x;u)$ is the perfect arc of $B_1$.
\item The endpoints of $\sigma(\cdot, s;u)$ lie on $\pl\Omega_s$:
\begin{align}\label{endpoints}
\sigma(0,s;u) = X(u_+(s),s), \quad \sigma(1,s;u) = X(u_-(s),s)
\end{align}
\item $\sigma(\cdot,s)$ encloses area $a$ together with $\pl\Omega_s$.
\end{enumerate}

Let
\begin{align*}
\frac{\pl X}{\pl s} = f \mbf N + g \mbf T
\end{align*}
be the variational field of $X(u,s)$
and
\begin{align*}
\frac{\pl\sigma}{\pl s} = \eta \mbf n + \xi \mbf t
\end{align*}
be the variational field of the arcs. To simplify notation, we write
\begin{align*}
\eta_0(x) = \eta(x,0), f_0(u) = f(u,0)
\end{align*}
and $\xi_0(x) = \xi(x,0), g_0(u) = g(u,0)$.

\subsection{First variations}
Since $\mathbf{t}(0) = - \mathbf{N}(u_+), \mathbf{t}(1) = \mathbf{N}(u_-), \mathbf{n}(0) = \mathbf{T}(u_+), \mathbf{n}(1) = \mathbf{T}(u_-)$, differentiating \eqref{endpoints} yields
\begin{align}\label{endpoints diff}
\eta_0(0) = \dot u_+ + g_0, \xi_0(0) = -f_0(u_+)
\end{align}
and similarly
\begin{align*}
\eta_0(1) = - \dot u_- - g_0, \xi_0(1) = f_0(u_-)
\end{align*}
where $\dot{u}_\pm = \frac{\pl}{\pl s}\Big|_{s=0} u_\pm.$

By assumption, the first variation of the area enclosed by the arc is zero
\begin{align}
\frac{\pl A}{\pl s} \Big|_{s=0} = \int_{u_-}^{u_+} f_0 \lt| \frac{\pl X}{\pl u}\rt| du + \int_0^1 \eta_0 \lt| \frac{\pl\sigma}{\pl x} \rt| dx =0. \label{little area preserving}
\end{align}
Together with the evolution of arc length (see (3.3) of \cite{AB})
\begin{align}
\frac{\pl}{\pl s} \lt| \frac{\pl\sigma}{\pl x} \rt| = \eta k  \lt| \frac{\pl\sigma}{\pl x}\rt| + \frac{\pl\xi}{\pl x},
\end{align}
we obtain the first variation of the length of the arc 
\begin{align}
\frac{\pl L}{\pl s} \Big|_{s=0} &= \int_0^1 \eta_0 k \lt| \frac{\pl\sigma}{\pl x} \rt| \,dx + \xi_0(1) - \xi_0(0) \\
&= -k \int_{u_-}^{u_+} f_0 \,du + f_0(u_-) + f_0(u_+) =: l(u)
\end{align}
noting that it does not depend on how the arc moves. 

At $s=0$, we have $u_+ = u_- + 2b$, $k = \cot b$ and hence
\[ l(u) = -\cot b \int_u^{u+2b} f_0(\tilde u)\,d\tilde u + f_0(u) + f_0(u+2b). \]
Since $X$ preserves enclosed area, we have $\int_0^{2\pi} f\,du=0$ and hence $\int_0^{2\pi} l \,du =0$.

If $l(u)$ is not identically zero, then there exists a prefect arc with $\frac{\pl L}{\pl s} \Big|_{s=0} < 0$ and the isoperimetric profile decreases. However, it is shown in Appendix A that there exists nontrivial variation such that $l(u) \equiv 0$. Therefore, we resort to the second variations.

\subsection{Second variations}
Firstly, since $\sigma(x,s)$ always encloses area $a$, we further differentiate \eqref{little area preserving} to get
\begin{align}\label{little area2}
\begin{split}
0 &= \frac{\pl}{\pl s}\Big|_{s=0} \lt( \int_0^1 \eta \lt| \frac{\pl\sigma}{\pl x}\rt| \,dx + \int_{u_-}^{u_+} f \lt| \frac{\pl X}{\pl u}\rt| \,du \rt)\\
&= \int_0^1 \frac{\pl\eta}{\pl s}\Big|_{s=0} \lt| \frac{\pl\sigma}{\pl x}\rt| + \eta_0 \lt( \eta_0 k \lt| \frac{\pl\sigma}{\pl x}\rt| + \frac{\pl\xi_0}{\pl x} \rt) \,dx + f_0(u_+) \dot{u}_+ - f_0(u_-)\dot{u}_- \\
&\quad + \int_{u_-}^{u_+} \frac{\pl f}{\pl s}\Big|_{s=0} + f(f + \frac{\pl g_0}{\pl u}) \,du.
\end{split}
\end{align}
By the evolution of unit tangent and normal vector of $\sigma$ (see page 514 of \cite{AB})
\begin{align*}
\frac{\pl \mathbf{t}}{\pl s} \Big|_{s=0} = \tilde\vphi \mathbf{n},\\
\frac{\pl\mathbf{n}}{\pl s} \Big|_{s=0} = -\tilde\vphi \mathbf{t}  
\end{align*}
with
\begin{align*}
\tilde\vphi = \frac{\frac{\pl \eta_0}{\pl x}}{\lt| \frac{\pl\sigma}{\pl x} \rt|} - k \xi_0,
\end{align*}
we get the second derivative of arc length
\begin{align*}
\frac{\pl^2}{\pl s^2}\Big|_{s=0} \lt| \frac{\pl\sigma}{\pl x}\rt| = \frac{\pl\eta}{\pl s} k \lt| \frac{\pl\sigma}{\pl x} \rt| - \frac{\pl}{\pl x} (\eta_0 \tilde\vphi) + \frac{\pl}{\pl x} \lt( \frac{\pl\xi}{\pl s}\Big|_{s=0} \rt) + \tilde\vphi \frac{\pl\eta_0}{\pl x}.
\end{align*}
Combining it with \eqref{little area2} we obtain the second variation of the length of the arc
\begin{align*}
\frac{\pl^2 L}{\pl s^2}\Big|_{s=0} &= \int_0^1 - k^2 \eta_0^2 \lt| \frac{\pl\sigma}{\pl x} \rt| + \frac{\lt( \frac{\pl\eta_0}{\pl x}\rt)^2}{\lt| \frac{\pl\sigma}{\pl x}\rt|} \,dx \\
&\quad + k \lt[ -\eta_0\xi_0 \Big|_{x=0}^{x=1}  - f_0(u_+)\dot{u}_+ + f_0(u_-)\dot{u}_- - \int_{u_-}^{u_+} \frac{\pl f}{\pl s}\Big|_{s=0} + f(f+ \frac{\pl g_0}{\pl u}) \,du  \rt]\\
&\quad + k \lt[ \frac{\pl\xi}{\pl s}\Big|_{s=0} - \eta_0 \tilde\vphi \rt]_{x=0}^{x=1}.
\end{align*}
To simplify the last line, we differentiate \eqref{endpoints} in $s$ twice 
\begin{align*}
&\frac{\pl\eta}{\pl s} \mathbf{n} - \eta_0 \tilde\vphi \mathbf{t} + \frac{\pl\xi}{\pl s} \mathbf{t} + \xi_0 \tilde\vphi \mathbf{n}\\
&= \frac{\pl^2 X}{\pl s \pl u} \dot{u}_\pm + \frac{\pl X}{\pl u} \ddot{u}_\pm  + \frac{\pl f}{\pl s}\mathbf{N} - f(\frac{\pl f}{\pl u} - g_0)\mathbf{T} + \frac{\pl g}{\pl s}\mathbf{T} + g(\frac{\pl f}{\pl u} - g_0)\mathbf{N}.
\end{align*}
Since $\mathbf{t}(0) = - \mathbf{N}(u_+), \mathbf{t}(1) = \mathbf{N}(u_-), \mathbf{n}(0) = \mathbf{T}(u_+), \mathbf{n}(1) = \mathbf{T}(u_-)$, we get
\begin{align*}
\frac{\pl\xi}{\pl s}\Big|_{s=0} - \eta_0 \tilde\vphi = - \eta_0 (\frac{\pl f_0}{\pl u} - g_0)(u_+) - \frac{\pl f}{\pl s}\Big|_{s=0}(u_+)
\end{align*}
at $x=0$ and
\begin{align*}
\frac{\pl\xi}{\pl s}\Big|_{s=0} - \eta_0 \tilde\vphi = - \eta_0(\frac{\pl f}{\pl u} - g_0)(u_-) + \frac{\pl f}{\pl s}\Big|_{s=0}(u_-)
\end{align*}
at $x=1$.
On the other hand, since $\Omega_s$ has area $\pi$, we have
\begin{align}\label{large area2}
0 = \frac{\pl}{\pl s}\Big|_{s=0} \int_0^{2\pi} f \lt|\frac{\pl X}{\pl u} \rt|\,du = \int_0^{2\pi} \frac{\pl f}{\pl s}\Big|_{s=0} + f_0(f_0 + \frac{\pl g_0}{\pl u}) \,du.
\end{align}

Next, to kill the effect of rigid motions in $\R^2$, we require that $X(u,s)$ is a normal variation in the first approximation:
\begin{align*}
g_0(u) \equiv 0.
\end{align*}
Moreover, we move the arcs tangentially in the first approximation:
\begin{align*}
\eta_0(x) \equiv 0.
\end{align*}
Putting these together, we obtain
\begin{align*}
\frac{\pl^2 L}{\pl s^2}\Big|_{s=0} = -k \int_{u_-}^{u_+} \lt( \frac{\pl f}{\pl s}\Big|_{s=0} + f^2 \rt) \,du + \frac{\pl f}{\pl s}\Big|_{s=0}(u_-) + \frac{\pl f}{\pl s}\Big|_{s=0}(u_+).
\end{align*}
Recall that we fix one arc in the above calculations. Taking all arcs into account, we obtain, by \eqref{large area2}, 
\begin{align*}
\int_0^{2\pi} \frac{\pl^2 L}{\pl s^2}\Big|_{s=0}(u) \,du = -2 \int_0^{2\pi} f_0^2(u)\,du <0
\end{align*}
for nontrivial variations. Hence, at least one arc becomes shorter while enclosing the same amount of area in the variation. In summary, we prove

\begin{theorem}\label{perturbation}
Let $\Omega_s, 0\le s < \epsilon$ be a family of domains with $\Omega_0 = B_1$ and $\mbox{area}(\Omega_s) = \pi$. We assume that $\Omega_s$ does not arise from rigid motion. Then for any $0 < A < \pi$ the following dichotomy on the perfect arcs that minimizes $I_{B_1}(A)$ holds. 
\begin{enumerate}
\item There exists a perfect arc whose length satisfies $\frac{d}{ds}\big|_{s=0} L < 0.$
\item We have $\frac{d}{ds}\big|_{s=0} L =0$ for all perfect arcs and there exists a perfect arc whose length satisfies $\frac{d^2}{ds^2}|_{s=0} L <0$ under tangential variation.
\end{enumerate}     
\end{theorem}
 
Consequently, the isoperimetric profile must decrease. This completes the proof of Theorem \ref{perturbation short}.

\appendix

\section{Local existence of perfect arcs}
Given two points $C_1$ and $C_2$ on a curve $C \subset \R^2$, denote the unit tangent and unit normal vector at $C_1$ and $C_2$ by $T_1,T_2$ and $N_1,N_2$. We observe an elementary criterion for two points to be joined by a perfect arc.
\begin{prop}\label{Condition_of_two_endpoints}
For two points $C_1 = C(s_1),C_2=C(s_2)$ on a curve $C$, a necessary and sufficient condition that there is a perfect arc passing through $C_1$ and $C_2$ is either the function $f: C \times C \rightarrow \mathbb{R}$ satisfies
\begin{align*}\label{f=0}
    f(s_1,s_2):=(C_1-C_2)\cdot(N_1+N_2) =0 \mbox{ with } N_1+N_2\neq0
\end{align*} or 
there is a straight line passing through $C_1$ and $C_2$ with direction $N_1=-N_2$.
\end{prop}
\begin{proof}
The assertion follows by observing that 
the angle between $C_1-C_2$ and $T_1$ is equal to the angle between $C_1-C_2$ and $T_2$.
\end{proof}

\begin{prop}
Let $\gamma$ be a perfect arc of $C$ with nonzero curvature and $C_1$ and $C_2$ be its endpoints. Denote the curvature of $C$ at $C_1$ and $C_2$ by $k_1$ and $k_2$. Then there exists a nontrivial family of perfect arcs $\gamma(t), -\delta < t < \delta$ such that $\gamma(0)=\gamma$ unless 
\begin{align*}
    k_1(C_1-C_2) = N_2 - N_1= k_2(C_1-C_2).
\end{align*} 
\end{prop}

\begin{proof}
Consider the two-point function $f: C \times C \rightarrow \mathbb{R}$ defined in the previous proposition. The partial derivatives of $f$ at $(C_1,C_2)$ are given by
\begin{align*}
    \frac{\partial f}{\partial s_1} &= T_1 \cdot N_2 - (C_1 - C_2) \cdot k_1T_1 \\
    \frac{\partial f}{\partial s_2} &= - T_2 \cdot N_1 - (C_1 - C_2) \cdot k_2T_2 
\end{align*}
We claim that  $\frac{\partial f}{\partial s_1}(C_1,C_2)$ and $\frac{\partial f}{\partial s_2}(C_1,C_2)$ do not vanish at the same time. Indeed, if $\frac{\partial f}{\partial S_1} (C_1,C_2) = \frac{\partial f}{\partial s_2}(C_1,C_2)=0$, then \[ N_2 - k_1(C_1-C_2) = \alpha N_1, \quad N_1 + k_2(C_1 - C_2) = \beta N_2\] for some constants $\alpha, \beta$. Since $\gamma$ has nonzero curvature, $N_1 + N_2 \neq 0$. Taking inner product with $N_1+N_2$, we get $\alpha=\beta =1$ and hence
\[ k_1(C_1-C_2) = N_2 - N_1= k_2(C_1-C_2).\]

Without loss of generality, we assume $\frac{\partial f}{\partial s_2}(C_1, C_2) \neq 0$. By the implicit function theorem, there is a function $g(s_1)$ defined on some open interval such that $f(s_1, g(s_1))=0$. By Proposition \ref{Condition_of_two_endpoints}, we obtain a family of perfect arcs. 
\end{proof}

We now present a local existence result of perfect arcs around a vertex. It is reminiscent of the existence of constant mean curvature foliation in a neighborhood of a point \cite{Ye}.
\begin{proposition}
Suppose there is a family of perfect arcs shrinking to $p$. Then $p$ must be a vertex of $C$. Conversely, if $p \in C$ is a non-degenerated
vertex ($k'=0$ but  $k''\neq 0$), then there is a family of perfect arcs shrinking to $p$.
\end{proposition}

\begin{proof}
We parametrize $C$ by arclength $s$. Suppose $p = C(0) =(0,0) $. The local canonical form of plane curves \cite[Section 1.6]{doCarmo} says
\begin{align*}
C(s) &= \lt(s -  \frac{k^2 s^3}{3!} \rt)t + \lt( \frac{s^2 k}{2} + \frac{s^3 k'}{3!} \rt)n + O(s^4), \\
t(s) &= \lt(1 - \frac{k^2s^2}{2} \rt)t + \lt( ks + \frac{k's^2}{2} \rt)n + O(s^3),\\
n(s) &= \lt(1 - \frac{k^2s^2}{2} \rt)n - \lt( ks + \frac{k's^2}{2} \rt)t + O(s^3) 
\end{align*}
where all terms on the right-hand side of the following equation are evaluated at $s=0$.

Suppose $\gamma_t, 0 <t<\epsilon$ is a family of perfect arcs shrinking to $p$ as $t \rw 0$. Let $C_1 = C(s_1), C_2 = C(s_2)$ be the endpoints of $\gamma$ where $s_1$ and $s_2$ depend on $t$ smoothly. Without loss of generality, we assume $s_1(t) = t$ and use $s_1$ as the parameter of the family $\gamma_t$; moreover, we assume $s_2 = O(s_1)$ as $s_1 \rw 0$. Recall that $f(s_1,s_2) = (C_1 - C_2)\cdot (N_1 + N_2) =0$. We compute the expansion of $f(s_1,s_2)$ with respect to $s_1$ to get
\begin{align*}
0 &= (C(s_1) - C(s_2))(n(s_1) - n(s_2))\\ &= - \frac{k'}{6}(s_1 - s_2)^3 + O(s_1^4).
\end{align*}  
By the assumption of $s_2$, $k'=0$ and hence $p$ is a vertex of $C$. 

For the converse, suppose $p$ is a non-degenerated vertex. We expand $\alpha(s)$ to higher order:
\begin{align*}
C(s) = \lt( s - \frac{k^2}{6}s^3 + \frac{k^4-4kk''}{120} s^5 \rt)t + \lt( \frac{k}{2}s^2 + \frac{k''-k^3}{24} s^4 + \frac{k'''}{120} s^5 \rt) + O(s^6)
\end{align*}
and it follows that
\begin{align*}
n(s) = \lt( 1 - \frac{k^2}{2}s^2 + \frac{k^4 - 4kk''}{24} s^4 \rt)n - \lt( ks + \frac{k''-k^3}{6}s^3 + \frac{k'''}{24} s^4 \rt)t + O(s^5).
\end{align*}
\\
\\
\noindent By direct computation,
$0 = f(s_1,s_2) = -\frac{k''}{12}(s_1^2 - s_2^2)(s_1 - s_2)^2 + O(s_1^5).$ 
\\
Since $k''\neq 0$, we obtain $a_1 = -1.$ We have in the next order
\begin{align*}
0 &= f(s_1,s_2) \\
&=(s_1-s_2) \lt[ (s_1 + s_2)A -\frac{k'''}{24}\lt( s_1^4 + s_2^4 \rt) + \frac{k'''}{60}\lt( s_1^4 + s_1^3s_2 + s_1^2s_2^2 + s_1s_2^3 + s_2^4 \rt)\rt] + O(s_1^6)
\end{align*}
where
\[ A =  -\frac{k''-k^3}{6}(s_1^2 - s_1 s_2 + s_2^2) + \frac{k^3}{6}(s_1^2 + s_1s_2 + s_2^2) + \lt( \frac{k''}{12} - \frac{k^3}{3} \rt)\lt( s_1^2 + s_2^2\rt). \]
Plugging in $s_2 = -s_1 + a_2 s_1^2 + O(s_1^3)$, the term in the bracket simplifies to
\begin{align*}
s_1^4 \lt( -\frac{k''}{3} a_2 - \frac{1}{15} k''' \rt) + O(s_1^5).
\end{align*}
Since $k''\neq 0$, for $s_1$ sufficiently small we can find two points $q^+, q^-$ on $\pl\Omega$  such that $f(C_1,q^+)>0$ and $f(C_1, q^-) <0$. For example, take $s_2 = -s_1 - \frac{k''' \pm 1}{5k''}s_1^2$. By the intermediate value theorem, there exists a point $q$ between $q^+$ and $q^-$ satisfying $f(C_1,q) =0$. $C_1$ and $q$ would give us a perfect arc. 
\end{proof}

We close this section by commenting the hypothesis on the number of vertices in Theorem \ref{bi-symmetric}. The above proposition shows that additional vertex may lead to additional minimizers of the variational problem \eqref{varitional problem}. Nonuniqueness of minimizers then cause isoperimetric profile non-differentiable. In general the isoperimetric profile of a compact Riemannian manifold is at best {\it piecewise} differentiable, see \cite{Bayle, GNP} for example.
 
\section{Perturbations that preserve the isoperimetric profile of $B_1$ in the first variation}
In Section 4.1 we show that a periodic function $f:[0,2\pi] \rw \R$ satisfying 
\[ l(u) := -\cot b \int_u^{u+2b} f(\tilde u)\,d\tilde u + f(u) + f(u + 2b) \equiv 0 \]
gives rise to a perturbation that causes isoperimetric profile of $B_1$ to have zero first variation. It is elementary to show that the relation holds for $\cos u$ and $\sin u$, which corresponds to translation of unit circle. For the perturbation to preserve enclosed area $\pi$, we also require $\int_0^{2\pi} f \,du =0$. To construct nontrivial perturbations, we consider Fourier series $f(\theta) = \sum'_n c_n e^{in\theta}$ with $c_n \in \R$. Here $\sum'$ means summation from $-\infty$ to $+\infty$ except $n =0.$ Direct computation yields
\[ l(\theta) = \sum\nolimits' \frac{c_n e^{in\theta}}{in} \lt( -\cos b (e^{inb}-1) + in \sin b (1 + e^{inb}) \rt) \]
and
\begin{align*} &-\cos b (e^{inb}-1) + in \sin b (1 + e^{inb})\\
&= 2 \sin nb (\cos b \sin nb - n \sin b \cos nb) + i \cdot 2 \cos nb (- \cos b \sin nb + n \sin b \cos nb)
\end{align*}
The figure below shows the implicit equation $\cos y \sin xy - x \sin y \cos xy =0$. Except the vertical lines $x=\pm 1, 0$ (they correspond to translations), every intersection of the curve with $y = n, n \in \mathbb{Z}$ gives rise to a nontrivial $f$ with $l(u) \equiv 0$.

\begin{figure}[h]
\centering
\includegraphics[scale=0.4]{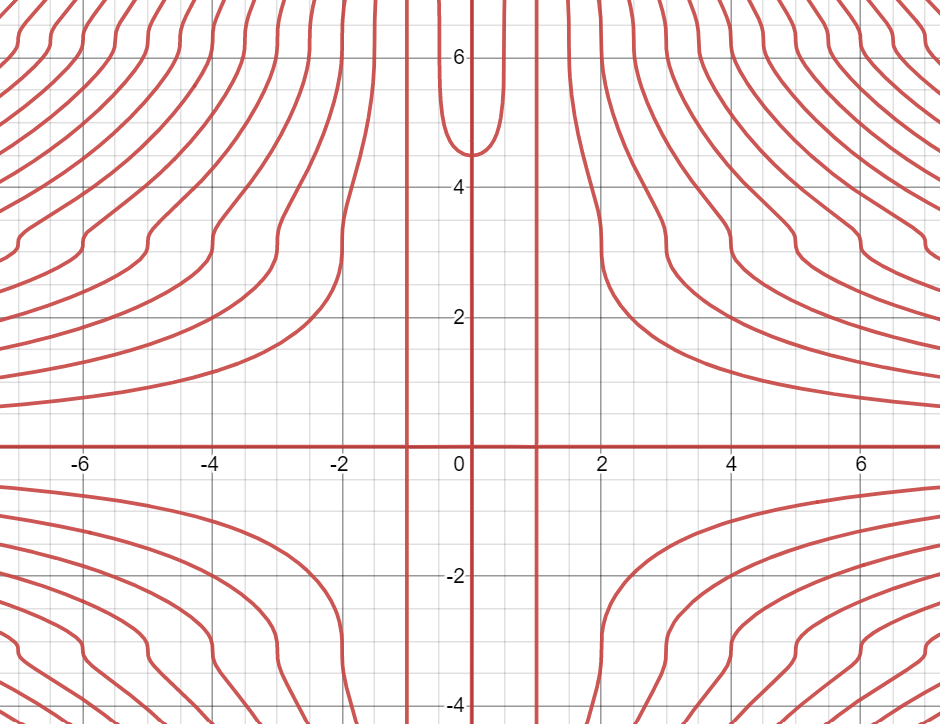}
\end{figure}


\bibliographystyle{amsplain}

\end{document}